\documentclass{amsart}
\usepackage{amsmath}
\usepackage{amsfonts}
\usepackage{amssymb}
\usepackage{latexsym}
\usepackage{graphicx}
\usepackage[final]{hyperref}
 \hypersetup{colorlinks=true, linkcolor=blue,
anchorcolor=blue, citecolor=red, filecolor=blue, menucolor=blue,
pagecolor=blue, urlcolor=blue}

\newcommand{\init}{\vert_{t = 0}}

\newcommand{\abs}[1]{\vert #1 \vert}

\newcommand{\norm}[1]{\left\Vert #1 \right\Vert}

\newcommand{\bignorm}[1]{\bigl\Vert #1 \bigr\Vert}

\newcommand{\C}{\mathbb{C}}

\newcommand{\R}{\mathbb{R}}

\newcommand{\innerprod}[2]{\langle \, #1 , #2 \, \rangle}

\newcommand{\angles}[1]{\langle #1 \rangle}

\newtheorem{theorem}{Theorem}[section]

\newtheorem{lemma}[theorem]{Lemma}

\theoremstyle{definition}

\theoremstyle{remark}

\numberwithin{equation}{section} 

\title[Uniqueness of solution for 2d DKG]{Unconditional uniqueness in the charge class for the Dirac-Klein-Gordon equations in two space dimensions}

\author[Selberg and Tesfahun]{Sigmund Selberg and Achenef Tesfahun}

\begin{document}

\maketitle

\begin{abstract}
Recently, A.~Gr\"unrock and H.~Pecher proved global well-posedness of
the 2d Dirac-Klein-Gordon equations given initial data for the
spinor and scalar fields in $H^s$ and $H^{s+1/2} \times H^{s-1/2}$, respectively,
where $s\ge 0$, but uniqueness was only known in a contraction space of Bourgain type, strictly smaller than the natural solution space $C([0,T];
H^s \times H^{s+1/2} \times H^{s-1/2})$. Here we prove uniqueness in the latter space for $s \ge 0$. This improves a recent result of H.~Pecher, where the range $s>1/30$ was covered.
\end{abstract}

\section{Introduction}\label{section_1}
 We consider the Dirac-Klein-Gordon system (DKG) in
two space dimensions, which reads
\begin{equation}\label{DKG0}
\left\{
\begin{aligned}
   -i\bigl( \partial_t + \alpha\cdot \nabla \bigr) \psi &= -M\beta\psi+ \phi \beta \psi,
    \\
(-\square+m^2) \phi &=  \innerprod{\beta \psi}{\psi}, \quad \quad &
( \square = -\partial_t^2 + \Delta)
\end{aligned}
\right.
\end{equation}
with initial data
\begin{equation}\label{Data0}
   \psi \init = \psi_0\in H^s, \qquad \phi \init =
\phi_0\in H^{r}, \qquad
\partial_t \phi \init = \phi_1 \in H^{r-1},
\end{equation}
where $\psi(t,x)$ is the Dirac spinor, regarded as a column vector
in $\C^2$, and $\phi(t,x)$ is real-valued. Here $t\in \R, \ x\in
\R^{2}$; $M,m \ge 0$ are constants; $\nabla=(\partial_{x_1},
\partial_{x_2}, )$; $ \innerprod{u}{v} :=
\innerprod{u}{v}_{\C^2}=v^\dagger u$  for column vectors $u, v \in
\C^2$, where $v^\dagger$ is the complex conjugate transpose of $v$;
$ H^s=(1+\sqrt{-\Delta})^{-s}L^2(\R^2)$ is the standard Sobolev
space of order $s$. The Dirac matrices $\alpha^j, \beta$ satisfy
\begin{equation}
  \label{DiracIdentity1}
  \beta^\dagger=\beta, \quad (\alpha^j)^\dagger = \alpha^j,
  \quad \beta^2 = (\alpha^j)^2 = I, \quad \alpha^j \beta + \beta \alpha^j
= 0.
  \end{equation}
  A particular representation is
$$
   \alpha^1 =    \begin{pmatrix}
     0 & 1  \\
     1 &0
   \end{pmatrix},
   \qquad
   \alpha^2 =  \begin{pmatrix}
     0 & -i  \\
     i & 0
   \end{pmatrix},
   \qquad
   \beta =    \begin{pmatrix}
     1 & 0  \\
     0 & -1
   \end{pmatrix}.
$$

Local well-posedness of this problem with low-regularity data was first studied in \cite{b2001}. In~\cite{dfs2006}, P.~D'Ancona, D.~Foschi and S.~Selberg proved the existence of a local solution
$$
(\psi, \phi, \partial_t \phi )\in C\left([0,T]; H^s \times H^r \times H^{r-1} \right)
$$
for the Cauchy problem \eqref{DKG0}, \eqref{Data0} whenever
$$ s>-\frac15, \quad
\max\left(\frac14+\frac{\abs{s}}2\right)<r<\min\left(\frac34+2s,
1+s\right),$$
but uniqueness was only obtained in the contraction space, which is strictly smaller than the natural solution space $C\left([0,T]; H^s \times H^r \times H^{r-1} \right)$. The question therefore arises whether uniqueness holds in the latter space; if this is the case, one says that \emph{unconditional uniqueness} holds at the regularity $(s,r)$.

The question of unconditional uniqueness is especially interesting for $s \ge 0 $ and $r=s+1/2$, since for that range
it was proved by A.~Gr\"unrock and H.~Pecher~\cite{gp2010} that the
local solution extends globally in time. Recently, H. Pecher ~\cite{p2011} proved unconditional uniqueness for $s > 1/30$, $r=s+1/2$. In this paper we improve this to $s \ge 0$.

\begin{theorem}\label{Thm-main}
If $s\ge 0$, uniqueness of the solution to \eqref{DKG0},
\eqref{Data0} holds in
$$(\psi, \phi, \partial_t\phi) \in C\left([0,T]; H^s \times H^{s+1/2}\times H^{s-1/2} \right).$$
\end{theorem}

To prove this we rely on the null structure of DKG, found in \cite{dfs2006}, and on bilinear space-time estimates of Klainerman-Machedon type proved in~\cite{dfs2010}. Using also an idea of Zhou~\cite{z2000} we iteratively improve the known regularity of the solution, until we reach a space where uniqueness is known by the results in~\cite{dfs2006}.

Some notation: In estimates we shall use $X \lesssim Y$ as shorthand for $X \le CY$,
where $C\gg 1$ is a constant. We write
$X\approx Y$ for $X\lesssim Y\lesssim X$. Throughout,
$\varepsilon$ is understood to be a sufficiently small positive
number. The space-time Fourier transform of a function $u(t,x)$ is denoted $\mathcal F u(\tau,\xi) = \widetilde u(\tau,\xi)$.

\section{Null structure and bilinear estimates}\label{section 2}
To uncover the null structure of the system we follow~\cite{dfs2006} and use the Dirac
projections
$$
P_{\pm}(\xi) = \frac{1}{2} \left( I \pm \frac{\xi}{\abs{\xi}} \cdot \alpha \right)
$$
to split $\psi = \psi_+  +  \psi_-$, where $\psi_\pm = P_\pm(D)
\psi$. Here $D = -i\nabla$, which has Fourier symbol $\xi$. Applying
$P_\pm(D)$ to the Dirac equation in \eqref{DKG0}, and using the
identities $\alpha \cdot D = \abs{D} P_+(D) - \abs{D}P_-(D)$,
$P^2_\pm(D)= P_\pm(D)$, $P_\pm(D)P_\mp(D)=0$ and
$P_\pm(D)\beta=\beta P_\mp(D)$,
 the Dirac equation becomes
  \begin{equation*}\label{Dirac}
   \bigl( -i\partial_t \pm\abs{D} \bigr) \psi_\pm = -M\beta\psi_\mp + P_\pm(D)(\phi \beta \psi).
   \end{equation*}
On the other hand, splitting $\phi =
\phi_+  +  \phi_-$, where $$\phi_\pm =\frac12\left(\phi\pm i
\angles{D}_m^{-1} \phi_t\right), \qquad  \angles{D}_m
=\sqrt{m^2+\abs{D}^2}, $$ the Klein-Gordon part of the system
becomes
\begin{equation*}\label{KG}
    \bigl( -i\partial_t \pm\angles{D}_m \bigr) \phi_\pm  = \mp(2\angles{D}_m)^{-1} \innerprod{\beta
  \psi}{\psi}.
   \end{equation*}
Thus, the DKG system has been rewritten as
\begin{equation}\label{DKG1}
\left\{
\begin{aligned}
  & \bigl( -i\partial_t \pm\abs{D} \bigr) \psi_\pm = -M\beta\psi_\mp + P_\pm(D)(\phi \beta \psi),
    \\
  & \bigl( -i\partial_t \pm\angles{D}_m \bigr) \phi_\pm  = \mp (2\angles{D}_m)^{-1} \innerprod{\beta
  \psi}{\psi},
\end{aligned}
\right.
\end{equation}
with initial data
\begin{equation}\label{Data1}
\psi_\pm(0)=P_\pm(D)\psi_0 \in  H^s, \qquad
\phi_\pm(0)=\frac12\left(\phi_0\pm i
\angles{D}_m^{-1}\phi_1\right)\in H^{s+\frac12}.
\end{equation}

Theorem \ref{Thm-main} then reduces to the following.
\begin{theorem}\label{Thm-main-1}
If $s\ge 0$, uniqueness of the solution to \eqref{DKG1},
\eqref{Data1} holds in
$$(\psi_\pm, \phi_\pm) \in C\left([0,T]; H^s \times H^{s+1/2}\right).$$
\end{theorem}

The null structure in the second equation in \eqref{DKG1} is quantified by the following estimate from ~\cite{dfs2006}: For Schwartz functions $\psi, \psi': \R^{1+2}\rightarrow \C^2$ and independent signs $\pm_1$ and $\pm_2$,
\begin{multline}\label{NullFormEstimate}
 \abs{\mathcal F\innerprod{\beta P_{\pm_1}(D) \psi}{P_{\pm_2}(D)
\psi'}(\tau,\xi)}
  \\
  \lesssim \int_{\R^{1+2}} \theta\left(\pm_1\eta,\pm_2(\eta-\xi)\right) \abs{\widetilde
\psi(\lambda,\eta)}
  \abs{\widetilde \psi'(\lambda-\tau,\eta-\xi)} \, d\lambda \, d\eta,
\end{multline}
where $\theta(\xi,\eta)$ denotes the angle between nonzero $\xi,\eta \in \R^2$. The null
structure in the first equation in \eqref{DKG1} is seen to be of the same type via a
duality argument (see the next section). To exploit this structure we need the following estimate:

\begin{lemma}\label{Lemma-angle}
Let $a, b, c \in [0,\frac12]$. For all signs
$(\pm_1,\pm_2)$, all $\lambda,\mu \in \R$ and all nonzero
$\eta,\zeta \in \R^2$,
\begin{equation}\label{AngleEst}
  \theta(\pm_1\eta,\pm_2\zeta)
  \lesssim
  \left(
  \frac{\angles{\abs{\lambda-\mu}-\abs{\eta-\zeta}}}
  {\min(\angles{\eta},\angles{\zeta})}
  \right)^{a}
  +
  \left(
  \frac{ \angles{\lambda\pm_1\abs{\eta}}}
  {\min(\angles{\eta},\angles{\zeta})}
  \right)^{b}+ \left(
  \frac{ \angles{
\mu\pm_2\abs{\zeta}}}
  {\min(\angles{\eta},\angles{\zeta})}
  \right)^{c}.
\end{equation}
\end{lemma}
\begin{proof}
In~\cite[Section 5.1]{dfs2007} it is shown that this holds when $a=b=c=1/2$, and since $\theta(\pm_1\eta,\pm_2\zeta)
  \lesssim 1$ the lemma follows.
\end{proof}

Next, we define some function spaces.

We remark that the symbols $\tau\pm \abs{\xi}$ and $\tau\pm
\angles{\xi}_m$ associated with the operators
$-i\partial_t \pm\abs{D}$ and $-i\partial_t \pm\angles{D}_m$ appearing in \eqref{DKG1}
are comparable, in the sense that $ \angles{\tau\pm
\abs{\xi}}\approx \angles{\tau\pm \angles{\xi}_m}$, where
$\angles{\cdot} =\sqrt{ 1 + \abs{\cdot}^2}$. Hence, the $X^{s,b}$ spaces
corresponding to the two operators are in fact equivalent. Specifically, the space $X^{s,b}_\pm$ in question is defined (given $s,b \in \R$) as the the
completion of $\mathcal{S}(\R^{1+2})$ with
respect to the norm
$$
   \norm{u}_{X^{s,b}_\pm} = \bignorm{\angles{\xi}^s
\angles{\tau\pm \abs{\xi}}^b \widetilde
u(\tau,\xi)}_{L^2_{\tau,\xi}}.
$$
We also need the wave-Sobolev space $H^{s,b}$ with norm
$$
  \norm{u}_{H^{s,b}} =
\bignorm{\angles{\xi}^s \angles{\abs{\tau}- \abs{\xi}}^b \widetilde
u(\tau,\xi)}_{L^2_{\tau,\xi}}.
$$
           For $T>0$, let
$X_\pm^{s,b}(S_T)$ and $H^{s,b}(S_T)$ be the respective restriction
spaces to the slab $S_T = (0,T) \times \R^2$.

Clearly, for $b \ge 0$,
\begin{equation}\label{X-H}
 \norm{u}_{H^{s,b}}\le
     \norm{u}_{X_\pm^{s,b}}.
\end{equation}

We also recall that for $b > 1/2$, $H^{s,b}$ is a proper subspace of $C(\R;H^s)$, and the inclusion is continuous. Thus,
$$X_\pm^{s,b} \subset H^{s,b} \subset C(\R,H^s) \quad \text{for $b > 1/2$}.$$

Moreover, it is well known that the
initial problem
$$
\bigl( -i\partial_t \pm \abs{D} \bigr)u =F\in X_\pm^{s, b-1}, \quad
 \qquad u|_{t=0}=f\in H^s,
 $$
 for any $s \in \R$ and $b > \frac12$, has a unique solution $u \in C\left([0,T]; H^s \right)$, and $u$ satisfies
\begin{equation}\label{LinearEst}
\norm{u}_{X^{s,b}_\pm(S_T) }\le C \Bigl( \norm{f}_{H^s}+
\norm{F}_{X^{s,b-1}_\pm(S_T)}\Bigr),
\end{equation}
where $C$ depends on $b$ and $T$. See, e.g., \cite{dfs2006}.

The final key ingredient that we need for the proof of uniqueness is the following product estimate
for the spaces $H^{s,b}$.

\begin{theorem}\label{Thm-prod}
Assume
\begin{align}
  \label{P1}
  &b_0\le 0< b_1, b_2
  \\
  \label{P2}
  &b_0+b_1 + b_2 > \frac12
  \\
  \label{P3}
  &b_0+b_1 > 0, \quad b_0+b_2 > 0
  \\
  \label{P4}
  &s_0 + s_1 + s_2 > \frac32 - (b_0+ b_1 + b_2)
  \\
  \label{P5}
  &s_0 + s_1 + s_2 > 1-(b_0+ b_1)
  \\
  \label{P6}
  &s_0 + s_1 + s_2 > 1 -(b_0+ b_2)
  \\
  \label{P7}
  &s_0 + s_1 + s_2 > \frac12 -b_0
  \\
  \label{P8}
  &s_0 + s_1 + s_2 > \frac34
  \\
  \label{P9}
  &s_0+b_0 + 2(s_1 + s_2) > 1
      \\
  \label{P10}  &s_1 + s_2 \ge -b_0, \quad  s_0 + s_2 \ge 0, \quad
   s_0 + s_1 \ge 0.
\end{align}
 Then the product estimate
\begin{equation}\label{ProdEst2}
   \norm{uv}_{H^{-s_0,-b_0}} \lesssim  \norm{u}_{H^{s_1,b_1}}\norm{v}_{H^{s_2,b_2}}
\end{equation}
holds for all $u, v \in \mathcal{S}(\R^{1+2})$.
\end{theorem}

This follows from Theorems 4.1 and 6.1 in \cite{dfs2010}.

\section{Proof of Uniqueness}\label{section3}
We now prove Theorem \ref{Thm-main-1}. Without loss of generality take $s=0$. From ~\cite{dfs2006} it is known that the solution is unique in the
iteration space
 $$(\psi_\pm, \phi_\pm)\in X_\pm^{-\frac5{32}-,\frac12+}(S_T)\times X_\pm^{\frac{11}{32}-,\frac12+}(S_T),$$
where we use the notation $a\pm:=a\pm \delta$ for
sufficiently small $\delta>0$.
 Thus, it suffices to show that if
\begin{equation}\label{Class-space}
 (\psi_\pm, \phi_\pm)\in C\left([0,T], L^2\times H^\frac12\right),
\end{equation}
is a solution of
 \eqref{DKG1}, \eqref{Data1}, then
 \begin{equation}\label{Unq-space}
(\psi_\pm, \phi_\pm)\in X_\pm^{-\frac5{32}-,\frac12+}(S_T)\times
X_\pm^{\frac{11}{32}-,\frac12+}(S_T).
\end{equation}

To this end, we split $\psi_\pm$ as
$$\psi_\pm=\psi_\pm^{\text{h}}+\psi_\pm^{\text{l}}+\Psi_\pm,$$
where $\psi_\pm^{\text{h}}$ is the homogeneous part while
$\psi_\pm^{\text{l}}$ and $\Psi_\pm$ are the inhomogeneous parts
corresponding to the linear and bilinear terms, respectively, in the right-hand
side of the first equation in \eqref{DKG1}. Similarly, we split
 $\phi_\pm$ as
$$\phi_\pm=\phi_\pm^{\text{h}}+\Phi_\pm,$$
where $\phi_\pm^{\text{h}}$ and $\Phi_\pm$ are the homogeneous and
inhomogeneous parts of $\phi_\pm$.

First note that by \eqref{LinearEst} and assumption \eqref{Class-space},
 $$\psi_\pm^{\text{h}}
 \in X_\pm^{0,\frac12+}(S_T) \quad \text{ and } \quad \phi_\pm^{\text{h}}\in
X_\pm^{\frac12,\frac12+}(S_T).$$
Moreover, \eqref{Class-space} implies
$$
  \psi_\pm \in X_\pm^{0,0}(S_T),
$$
so \eqref{LinearEst} gives
 $$\psi_\pm^{\text{l}}
 \in X_\pm^{0,1}(S_T).$$ 

Thus it only remains to show that the pair
$\left(\Psi_\pm,\Phi_\pm\right)$ satisfies \eqref{Unq-space}. To
this end we start with \eqref{Class-space} and
use the null structure and product estimates to successively improve the regularity.

Note that by the above, $\Psi_\pm = \psi_\pm - \psi_\pm^{\text{h}} - \psi_\pm^{\text{l}} \in X_\pm^{0,0}(S_T)$, and a
similar argument shows that $\Phi_\pm = \phi_\pm - \phi_\pm^{\text{h}} \in X_\pm^{\frac12,0}$.

\subsection{First estimate for $\Psi_\pm$}
We first claim
\begin{equation}\label{psi:01}
\Psi_\pm \in X_\pm^{-\frac12\alpha, \alpha}(S_T) \quad \text{for} \
\alpha\in [0,1].
\end{equation}
Indeed, using \eqref{LinearEst},
\begin{align*}
\norm{\Psi_\pm}_{ X_\pm^{-\frac12,1}(S_T) } \lesssim
\norm{\phi\beta\psi}_{ X_\pm^{-\frac12, 0}(S_T)}
&=\norm{\phi\beta\psi}_{L_t^{2}H_x^{-\frac12}(S_T)}\\
&\lesssim T^\frac12 \norm{\phi}_{ L_t^\infty
H_x^\frac12(S_T)}\norm{\psi}_{L_t^\infty L_x^2(S_T)},
\end{align*}
where in the last line we used the product Sobolev inequality in the
$x$-variable (see, e.g., the introduction in ~\cite{dfs2010}).
Interpolating $\Psi_\pm \in X_\pm^{-\frac12, 1}(S_T)$ with $\Psi_\pm
\in X_\pm^{0, 0}(S_T)$, we get \eqref{psi:01}.

\subsection{First estimate for $\Phi_\pm$ }\label{subsec-phi1st}
We show that
\begin{equation}\label{phi:01}
\Phi_\pm \in X_\pm^{\frac14-\varepsilon,\frac12+\varepsilon}(S_T).
\end{equation}
Using \eqref{psi:01} with $\alpha=\frac14+2\varepsilon$ and
\eqref{LinearEst}, the claim reduces to the bilinear estimate
\begin{equation*}
 \norm{ \innerprod{\beta P_{\pm_1}(D) \psi
}{P_{\pm_2}(D)
\psi'}}_{X_{\pm_0}^{-\frac34-\varepsilon,-\frac12+\varepsilon}}
\lesssim
 \norm{\psi}_{X_{\pm_1}^{-\frac18-\varepsilon,\frac14+2\varepsilon}}
\norm{\psi'}_{X_{\pm_2}^{-\frac18-\varepsilon,\frac14+2\varepsilon}}.
\end{equation*}

Without loss of generality, take $\pm_1=+, \pm_2=\pm$. By Plancherel
and using \eqref{NullFormEstimate} we reduce to
$$
I\lesssim
\norm{u}_{X_{+}^{-\frac18-\varepsilon,\frac14+2\varepsilon}}
\norm{v}_{X_{\pm}^{-\frac18-\varepsilon,\frac14+2\varepsilon}},
$$
where
$$ I  =  \norm{
  \iint
  \frac{ \theta(\eta, \pm(\eta-\xi)) }{\angles{\xi}^{\frac34+\varepsilon}\angles{\abs{\tau}-\abs{\xi}}^{\frac12-\varepsilon}}
  \widehat u(\lambda,\eta)
  \widehat v(\lambda-\tau,\eta-\xi) \ d\lambda \ d\eta
  }_{L^2_{\tau,\xi}}. $$

Applying Lemma \ref{Lemma-angle} with $a=\frac12-\varepsilon$ and
$b=c=\frac14+2\varepsilon$, and using also \eqref{X-H}, the estimate
for $I$ reduces to the following three estimates:
\begin{equation}\label{phi-1st}
\left\{
\begin{aligned}
  \norm{uv}_{H^{-\frac34-\varepsilon,0}}
  &\lesssim
  \norm{u}_{H^{\frac38-2\varepsilon, \frac14+2\varepsilon}}
  \norm{v}_{H^{-\frac18-\varepsilon,\frac14+2\varepsilon}}
  \\
  \norm{uv}_{H^{-\frac34-\varepsilon, -\frac12+\varepsilon}}
  &\lesssim
  \norm{u}_{H^{\frac18+\varepsilon, 0}}
  \norm{v}_{H^{-\frac18-\varepsilon, \frac14+2\varepsilon}}
  \\
  \norm{uv}_{H^{-\frac34-\varepsilon, -\frac12+\varepsilon}}
  &\lesssim\norm{u}_{H^{-\frac18-\varepsilon, 0}}
   \norm{v}_{H^{\frac18+\varepsilon, \frac14+2\varepsilon}}.
   \end{aligned}
\right.
\end{equation}
 All these hold by Theorem
 \ref{Thm-prod} (via duality for the last two), proving \eqref{phi:01}. Interpolating with $\Phi_\pm \in X_\pm^{\frac12,
0}$ we get moreover
\begin{equation}\label{phi:02}
\Phi_\pm \in
X_\pm^{\frac7{16}-\frac\varepsilon4,\frac18+\frac\varepsilon4}(S_T),
\end{equation}
which we now use to improve \eqref{psi:01}.

\subsection{Second estimate for $\Psi_\pm$}\label{subsec-psi2nd}
  We show that
  \begin{equation}\label{psi:02}
  \Psi_\pm \in X_\pm^{-\frac7{32},\frac12+\varepsilon}(S_T).
\end{equation}
Applying \eqref{LinearEst} and using \eqref{psi:01} (with
$\alpha=\frac12-2\varepsilon$) and \eqref{phi:02}, the claim reduces
to the bilinear estimate
\begin{equation*}\label{NLE-psi:01}
 \norm{P_{\pm_2}(D)(\phi \beta P_{\pm_1}(D)
\psi)}_{X_{\pm_2}^{-\frac7{32},-\frac12 +\varepsilon}} \lesssim
 \norm{\phi}_{X_{\pm_0}^{\frac7{16}-\frac\varepsilon4,\frac18+\frac\varepsilon4}}
\norm{\psi}_{X_{\pm_1}^{-\frac14+\varepsilon,\frac12-2\varepsilon}}.
\end{equation*}
Duality reduces this to
\begin{equation*}
\norm{ \innerprod{\beta P_{\pm_1}(D) \psi }{P_{\pm_2}(D)
\psi'}}_{X_{\pm_0}^{-\frac7{16}+\frac\varepsilon4,-\frac18-\frac\varepsilon4}}
\lesssim
 \norm{\psi'}_{X_{\pm_1}^{\frac7{32},\frac12 -\varepsilon }}
\norm{\psi}_{X_{\pm_2}^{-\frac14+\varepsilon,\frac12-2\varepsilon}}.
\end{equation*}
Taking again $\pm_1=+, \pm_2=\pm$ without loss of generality, applying
Plancherel and using \eqref{NullFormEstimate}, we reduce further to
$$
J\lesssim  \norm{u}_{X_{+}^{\frac7{32},\frac12 -\varepsilon}}
\norm{v}_{X_{\pm}^{-\frac14+\varepsilon,\frac12-2\varepsilon}},
$$
where
$$ J  =  \norm{
  \iint
  \frac{ \theta(\eta, \pm(\eta-\xi)) }{\angles{\xi}^{\frac7{16}-\frac\varepsilon4}\angles{\abs{\tau}-\abs{\xi}}^{\frac18+\frac\varepsilon4}}
  \widehat u(\lambda,\eta)
  \widehat v(\lambda-\tau,\eta-\xi) \ d\lambda \ d\eta
  }_{L^2_{\tau,\xi}}, $$

Applying Lemma \ref{Lemma-angle} with $a=\frac12$,
$b=\frac12-\varepsilon$ and $c=\frac12-2\varepsilon$, and using
\eqref{X-H}, we reduce to the following six
estimates:
\begin{equation}\label{psi-2nd}
\left\{
\begin{aligned}
  \norm{uv}_{H^{-\frac7{16}+\frac\varepsilon4,
  \frac38-\frac{\varepsilon}4}}
  &\lesssim
  \norm{u}_{H^{\frac{23}{32}, \frac12-\varepsilon}}
  \norm{v}_{H^{-\frac14+\varepsilon,\frac12-2\varepsilon}}
  \\
  \norm{uv}_{H^{-\frac7{16}+\frac\varepsilon4,
  \frac38-\frac\varepsilon4}}
  &\lesssim
  \norm{u}_{H^{\frac7{32}, \frac12-\varepsilon}}
  \norm{v}_{H^{\frac14+\varepsilon,\frac12-2\varepsilon}}
  \\
  \norm{uv}_{H^{-\frac7{16}+\frac\varepsilon4,-\frac18-\frac\varepsilon4}}
  &\lesssim
  \norm{u}_{H^{\frac{23}{32}-\varepsilon, 0}}
  \norm{v}_{H^{-\frac14+\varepsilon,\frac12-2\varepsilon}}
  \\
  \norm{uv}_{H^{-\frac7{16}+\frac\varepsilon4,-\frac18-\frac\varepsilon4}}
  &\lesssim\norm{u}_{H^{\frac7{32}, 0}}
   \norm{v}_{H^{\frac14, \frac12-2\varepsilon}}
  \\
  \norm{uv}_{H^{-\frac7{16}+\frac\varepsilon4,-\frac18-\frac\varepsilon4}}
  &\lesssim
  \norm{u}_{H^{\frac{23}{32}-2\varepsilon, \frac12-\varepsilon}}
  \norm{v}_{H^{-\frac14+\varepsilon, 0}}
  \\
  \norm{uv}_{H^{-\frac7{16}+\frac\varepsilon4,-\frac18-\frac\varepsilon4}}
  &\lesssim
  \norm{u}_{ H^{\frac7{32}, \frac12-\varepsilon } }
  \norm{v}_{ H^{\frac14-\varepsilon,0} }.
 \end{aligned}
\right.
\end{equation}
All these hold by Theorem \ref{Thm-prod}
(via duality for the last four), proving \eqref{psi:02}. Interpolation with
$\Psi_\pm \in X_\pm^{0, 0}$ yields
\begin{equation}\label{psi:03}
\Psi_\pm \in X_\pm^{-\frac7{64},\frac14+\frac\varepsilon2}(S_T).
\end{equation}
We now use this to improve \eqref{phi:01}.

\subsection{Second estimate for $\Phi_\pm$ }
We show that
\begin{equation}\label{phi:03}
\Phi_\pm \in
X_\pm^{\frac9{32}-\varepsilon,\frac12+\varepsilon}(S_T).
\end{equation}
Applying \eqref{LinearEst} and using \eqref{psi:03} we reduce to
\begin{equation*}
\label{NLE-phi:02} \norm{ \innerprod{\beta P_{\pm_1}(D) \psi
}{P_{\pm_2}(D)
\psi'}}_{X_{\pm_0}^{-\frac{23}{32}-\varepsilon,-\frac12+\varepsilon,}}
\lesssim
 \norm{\psi}_{X_{\pm_1}^{-\frac7{64},\frac14+\frac\varepsilon2}}
\norm{\psi'}_{X_{\pm_2}^{-\frac7{64},\frac14+\frac\varepsilon2}}.
\end{equation*}
Proceeding as in subsection \ref{subsec-phi1st}, applying Lemma
\ref{Lemma-angle} with $a=\frac12-\varepsilon$ and
$b=c=\frac14+\frac\varepsilon2$, and using \eqref{X-H}, we further reduce to the following three estimates:
\begin{equation}\label{phi-2nd}
\left\{
\begin{aligned}
  \norm{uv}_{H^{-\frac{23}{32}-\varepsilon, 0}}
  &\lesssim
  \norm{u}_{H^{\frac{25}{64}-\varepsilon, \frac14+\frac\varepsilon2}}
  \norm{v}_{H^{-\frac7{64},\frac14+\frac\varepsilon2}}
  \\
  \norm{uv}_{H^{-\frac{23}{32}-\varepsilon, -\frac12+\varepsilon} }
  &\lesssim
  \norm{u}_{H^{\frac{9}{64}+\frac\varepsilon2, 0}}
  \norm{v}_{H^{-\frac7{64},\frac14+\frac\varepsilon2}}
  \\
  \norm{uv}_{H^{-\frac{23}{32}-\varepsilon, -\frac12+\varepsilon}}
  &\lesssim\norm{u}_{H^{-\frac7{64}, 0}}
   \norm{v}_{H^{\frac{9}{64}+\frac\varepsilon2,\frac14+\frac\varepsilon2}}.
   \end{aligned}
\right.
\end{equation}
All these are seen to hold by Theorem \ref{Thm-prod} (via duality for the last two).

Interpolating \eqref{phi:03} with $\Phi_\pm \in X_\pm^{\frac12, 0}$,
we further get
\begin{equation}\label{phi:04}
\Phi_\pm \in
X_\pm^{\frac{25}{64}-\frac\varepsilon2,\frac14+\frac\varepsilon2}(S_T).
\end{equation}

\subsection{Third estimate for $\Psi_\pm$}
  We prove
  \begin{equation}\label{psi:04}
  \Psi_\pm \in X_\pm^{-\frac5{32}-3\varepsilon,\frac12+\varepsilon}(S_T).
\end{equation}
Applying \eqref{LinearEst} and using \eqref{psi:02} and
\eqref{phi:04}, the claim reduces (after duality) to
\begin{equation*}
\norm{ \innerprod{\beta P_{\pm_1}(D) \psi }{P_{\pm_2}(D)
\psi'}}_{X_{\pm_0}^{-\frac{25}{64}+\frac\varepsilon2,-\frac14-\frac\varepsilon2}}
\lesssim
 \norm{\psi'}_{X_{\pm_1}^{\frac5{32}+3\varepsilon,\frac12 -\varepsilon }}
\norm{\psi}_{X_{\pm_2}^{-\frac7{32},\frac12+\varepsilon}}.
\end{equation*}
Proceeding as in subsection \ref{subsec-psi2nd}, applying Lemma
\ref{Lemma-angle} with $a=\frac12$, $b=\frac12-\varepsilon$ and
$b=\frac12$, and using \eqref{X-H}, we reduce to
\begin{equation}\label{psi-3rd}
\left\{
\begin{aligned}
  \norm{uv}_{H^{-\frac{25}{64}+\frac\varepsilon2, \frac14-\frac\varepsilon2}}
  &\lesssim
  \norm{u}_{H^{\frac{21}{32}+3\varepsilon, \frac12-\varepsilon}}
  \norm{v}_{H^{-\frac7{32},\frac12+\varepsilon}}
  \\
  \norm{uv}_{H^{-\frac{25}{64}+\frac\varepsilon2,\frac14-\frac\varepsilon2}}
  &\lesssim
  \norm{u}_{H^{\frac5{32}+3\varepsilon, \frac12-\varepsilon}}
  \norm{v}_{H^{\frac9{32},\frac12+\varepsilon}}
  \\
  \norm{uv}_{H^{-\frac{25}{64}+\frac\varepsilon2,-\frac14-\frac\varepsilon2}}
  &\lesssim
  \norm{u}_{H^{\frac{21}{32}+2\varepsilon, 0}}
  \norm{v}_{H^{-\frac7{32},\frac12+\varepsilon}}
  \\
  \norm{uv}_{H^{-\frac{25}{64}+\frac\varepsilon2,-\frac14-\frac\varepsilon2}}
  &\lesssim\norm{u}_{H^{\frac5{32}+3\varepsilon, 0}}
   \norm{v}_{H^{\frac9{32}-\varepsilon, \frac12+\varepsilon}}
  \\
  \norm{uv}_{H^{-\frac{25}{64}+\frac\varepsilon2,-\frac14-\frac\varepsilon2}}
  &\lesssim
  \norm{u}_{H^{\frac{21}{32}+3\varepsilon, \frac12-\varepsilon}}
  \norm{v}_{H^{-\frac7{32}, 0}}
  \\
  \norm{uv}_{H^{-\frac{25}{64}+\frac\varepsilon2,-\frac14-\frac\varepsilon2}}
  &\lesssim
  \norm{u}_{ H^{\frac5{32}+3\varepsilon, \frac12-\varepsilon } }
  \norm{v}_{ H^{\frac9{32},0} },
 \end{aligned}
\right.
\end{equation}
all of which hold by Theorem
 \ref{Thm-prod} (via duality for the last four).

 \subsection{Third estimate for $\Phi_\pm$ }
We finally prove that
\begin{equation}\label{phi:05}
\Phi_\pm \in
X_\pm^{\frac{11}{32}-3\varepsilon,\frac12+\varepsilon}(S_T).
\end{equation}
By \eqref{LinearEst} and \eqref{psi:04} we reduce to
\begin{equation*}
 \norm{ \innerprod{\beta P_{\pm_1}(D) \psi
}{P_{\pm_2}(D)
\psi'}}_{X_{\pm_0}^{-\frac{21}{32}-3\varepsilon,-\frac12+\varepsilon,}}
\lesssim
 \norm{\psi}_{X_{\pm_1}^{-\frac5{32}-3\varepsilon,\frac12+\varepsilon}}
\norm{\psi'}_{X_{\pm_2}^{-\frac5{32}-3\varepsilon,\frac12+\varepsilon}}.
\end{equation*}
Proceeding as in subsection \ref{subsec-phi1st}, applying Lemma
\ref{Lemma-angle} with $a=\frac12-\varepsilon$ and $b=c=\frac12$,
the estimate reduces to
\begin{equation}\label{phi-3rd}
\left\{
\begin{aligned}
  \norm{uv}_{H^{-\frac{21}{32}-3\varepsilon, 0}}
  &\lesssim
  \norm{u}_{H^{\frac{11}{32}-4\varepsilon, \frac12+\varepsilon}}
  \norm{v}_{H^{-\frac5{32}-3\varepsilon,\frac12+\varepsilon}}
  \\
  \norm{uv}_{H^{-\frac{21}{32}-3\varepsilon, -\frac12+\varepsilon} }
  &\lesssim
  \norm{u}_{H^{\frac{11}{32}-3\varepsilon, 0}}
  \norm{v}_{H^{-\frac5{32}-3\varepsilon,\frac12+\varepsilon}}
  \\
  \norm{uv}_{H^{-\frac{21}{32}-3\varepsilon, -\frac12+\varepsilon}}
  &\lesssim\norm{u}_{H^{-\frac5{32}-3\varepsilon, 0}}
   \norm{v}_{H^{\frac{11}{32}-3\varepsilon, \frac12+\varepsilon}}.
   \end{aligned}
\right.
\end{equation}
All these are true by Theorem
 \ref{Thm-prod}.

 This finishes the proof of Theorem \ref{Thm-main-1}.

\end{document}